\documentclass{article}

\usepackage{amsmath}
\usepackage{amsthm}
\usepackage{amssymb}
\usepackage{bbm}
\usepackage{verbatim}
\usepackage{graphicx}
\usepackage{afterpage}
\usepackage{etoolbox}
\usepackage{float}
\usepackage{algorithmic}
\usepackage[algosection,algoruled,vlined]{algorithm2e}
\usepackage[noadjust]{cite}
\usepackage{url}
\usepackage{authblk}
\usepackage{hyperref}
\usepackage[utf8]{inputenc}
\usepackage{authblk}
\usepackage{graphicx}
\usepackage{caption}
\usepackage{subcaption}
\usepackage{multirow}
\usepackage{float}

\newtheorem{theorem}{Theorem}[section]

\newtheorem{lemma}[theorem]{Lemma}
\newtheorem{corollary}[theorem]{Corollary}
\newtheorem{rem}[theorem]{Remark}

\newcommand{\argmin}{\operatornamewithlimits{\text{argmin}}}

\def \< {\langle}
\def \> {\rangle}

\def \im {{\rm Im}}

\DeclareMathOperator{\rank}{rank}
\DeclareMathOperator{\nnz}{nnz}

\usepackage{csquotes}
\usepackage[a4paper, margin=1.1in]{geometry}
\begin{document}

\title{Randomized LU Decomposition Using Sparse Projections}

\author
{Yariv Aizenbud${^1}$~~Gil Shabat${^2}$~~Amir Averbuch${^2}$\\
${^1}$School of Applied Mathematics, Tel Aviv University, Israel\\
${^2}$School of Computer Science, Tel Aviv University, Israel
}

\maketitle

\begin{abstract}
    A fast algorithm for the approximation of a low rank LU decomposition is presented. In order to achieve a low complexity, the algorithm uses sparse random projections combined with FFT-based random projections. The asymptotic approximation error of the algorithm is analyzed and a theoretical error bound is presented. Finally, numerical examples illustrate that for a similar approximation error, the sparse LU algorithm is faster than recent state-of-the-art methods. The algorithm is completely parallelizable that enables to run on a GPU. The performance is tested on a GPU card, showing a significant improvement in the running time in comparison to sequential execution.
\end{abstract}


\section{Introduction}
Low-rank matrix  approximations are  a key component for efficient
processing, manipulating and analysis of big datasets. Often, data
matrices can be very large and yet have many redundancies and
dependencies between rows and columns that result in being a low-rank matrix. Finding a low-rank approximation of a matrix enables us to
process the entire matrix by using only a small set of vectors.
Applications that utilize low-rank matrix approximations include
data compression, noise filtering, principle component analysis
and kernel methods, to name some. Although a low-rank matrix
approximation can be computed using well-known matrix decomposition
methods, such as singular value decomposition (SVD) or rank
revealing QR (RRQR), very often this is impractical due to high
computational load. Therefore, there is an ongoing interest in the development of fast algorithms
for computing low-rank matrix approximations. Randomized algorithms
for low rank matrix approximations include SVD
\cite{halko2011finding,WLRT,randecomp}, LU
\cite{shabat2013randomized}, CUR
\cite{drineas2006fast,drineas2008relative}, principal component
analysis (PCA) \cite{szlam2014implementation,halko2011algorithm}, to
name some. Randomized algorithms have gained an increasing
popularity because of their abilities to perform matrix computations
faster and on larger data sets than classical algorithms such as
\cite{golub2012matrix}. 

Sparse random projections have been studied
for dimensionality reduction as a sparse variant of the
Johnson-Lindenstrauss (JL) transform. A nearly tight lower bound for
several dimensionality reduction linear maps for a predetermined 
sparsity is given in \cite{nelson2013sparsity,nelson2014lower}.

Algorithms, which utilize sparse random projections for SVD and
regression computations, are given in
\cite{clarkson2013low,kane2014sparser,nelson2013osnap,achlioptas2007fast}. Algorithms that are based on sparse dimensionality reduction transforms benefit from the fact that their projection step is more computationally efficient than those that use dense matrices in their projection step.
While the complexity of the algorithms, which use a structured JL transform such as FFT-based random projections \cite{WLRT}, does not change when applied to sparse matrices, algorithms that are based on sparse random projections are accelerated when applied to sparse matrices.

In this paper, the randomized LU algorithms
\cite{shabat2013randomized} are extended by utilizing sparse random projections. We introduce an LU decomposition algorithm that uses sparse random projections combined with the fast Johnson-Lindenstrauss (FJL) transform. FJL transforms are based on the fast Fourier transform (FFT) \cite{AC} and are also used in
\cite{WLRT}. This combination of sparse JL with FJL was introduced in \cite{clarkson2013low} to produce faster
algorithms. The algorithm presented in this paper is shown to be significantly
faster for a low-rank matrix decomposition than the algorithms mentioned above.
In addition, a detailed theoretical analysis is presented for the
derived error bounds of the algorithm.

For a given matrix $A$ of size $m \times n$, the algorithm computes
the lower and upper triangular matrices $L$ and $U$ of sizes $m
\times k$ and $k \times n$, respectively, and permutation matrices
$P$ and $Q$ such that with high probability
\begin{equation}
\|LU - PAQ\|_F \leq \mathcal{O}(\Delta_r)
\end{equation}
where $\Delta_r\stackrel{\Delta}{=}\sqrt{\sum_{i=r+1}^{\min (m,n)}
\sigma_i^2}$, $r<k$.
Then, the performance of the algorithm is compared with the current
state-of-the-art methods that compute low-rank matrix
approximations. The presented algorithm is parallelizable and can be
fully implemented on a GPU.

The paper is organized as follows: Section \ref{sec:preliminaries}
reviews some mathematical results that are needed for the
development of the sparse randomized LU algorithm. Section
\ref{sec:sparserandlu} presents the sparse randomized LU algorithm
and the error bound resulted from the approximation. Section
\ref{sec:numerical_results} presents numerical results for the
approximation  error and for the running time of the sparse
randomized LU with comparison to other algorithms.

\section{Preliminaries}
\label{sec:preliminaries} This section presents the mathematical
background needed in the rest of the paper.  More specifically, we
review the properties of the Sub-sampled Random Fourier Transform
(SRFT) matrices and the sparse embedding matrices. Throughout the
paper, $\Vert \cdot \Vert_F$ denotes the Frobenius norm, $\Vert
\cdot \Vert_2$ denotes the spectral norm when the argument is a
matrix or the $l_2$ (Euclidean) norm for vector arguments. $M_{m
\times n}$ is the set of $m \times n$ matrices, $\sigma_i(\cdot)$ is
the $i$th largest singular value of a matrix, and
$\Delta_k(\cdot)=\sqrt{\sum_{i=k+1}^{\min(m,n)}\sigma_i^2}$ ,
$k=0,\ldots,\min(m,n)-1$.

\subsection{The SRFT matrix} \label{subsec:srft}
The SRFT matrix, which is presented in \cite{AC,WLRT}, is a random matrix
denoted by $\Pi$. It is decomposed into $\Pi = DFS$ where $D$ is an
$n \times n$ diagonal matrix whose entries are i.i.d. random
variables drawn from a uniform distribution on the unit circle in
$\mathbb{C}$, $F$ is an $n\times n$ discrete Fourier transform such
that $F_{jk}= \frac{1}{\sqrt{n}}e^{-2\pi i
    (j-1)(k-1)/n}$, $j,k = 1, \ldots, n $ and $S$ is an $n\times l$ matrix whose entries are all
zeros except for a single randomly placed 1 in each column.

Lemma \ref{lem:SRFTmult} shows that matrix multiplication by an SRFT
matrix can be done faster in comparison to an arbitrary matrix.
\begin{lemma}[\cite{WLRT}] \label{lem:SRFTmult}
    For any $m \times n$ matrix $A$, let $\Pi$ be the $n \times l$  SRFT matrix. Then, $Y=A\Pi$ can be computed in $\mathcal{O}(mn\log l)$  floating point operations.
\end{lemma}

\begin{theorem}[Follows from 
    Theorem 1.3 in \cite{tropp2011improved} ] \label{thm:FJLT_SRFT}
    For any $U \in M_{n\times r}$ with orthogonal columns, if $\Pi \in M_{k \times n}$,
    is a randomly chosen SRFT matrix, where $r,k$ and $n$ satisfy
    $4\left[ \sqrt{r} + \sqrt{8\log(rn)}\right]^2 \log r \leq k \leq n$.  Then, with probability of at least $1-\mathcal{O}(r^{-1})$, the largest and the smallest singular values of $\Pi U $ are in $[0.40,1.48]$.
\end{theorem}

\subsection{Sparse Embedding Matrices} \label{subsec:SEM}

For a parameter $t \in \mathbb{N}$, consider the random linear map $S = \Phi D$, where $S \in M_{k \times n}$, such that for $h : \{1,\ldots, n\} \to \{1,\ldots, k\}$,a random map such that for each $i \in \{1,\ldots, n\}$, $h(i) = t'$ for $t' \in \{1,\ldots, k\}$ with probability $1/t$, we have
\begin{enumerate}
    \item $\Phi \in \{0,1\}^{k \times n}$ is a $k \times n$ ($k \le n$) binary matrix with nonzero entries $\Phi_{h(i),i} = 1$ and all the remaining entries equal to $0$. In other words, $\Phi$ is a matrix with a single 1 in each row.
    \item $D$ is an $n\times n$ random diagonal matrix where each diagonal entry  is independently chosen to be $+1$ or $-1$ with equal probability.
\end{enumerate}
 A matrix $S$ that satisfies 1 \& 2 is referred to as a sparse embedding matrix (SEM).
\begin{lemma}
    Let $S \in M_{k\times n}$ be an SEM matrix. Then, $\|S\|_F = \sqrt{n}$.
\end{lemma}
\begin{theorem}\label{thm:sing_vals_S}
    The largest singular value of a $k\times n$ SEM is bounded, with high probability, by $C(n,k) = \sqrt{\frac{n}{k} + \sqrt{2\frac{n}{k} \log k}}$ for large enough $n$.
\end{theorem}
The proof of Theorem \ref{thm:sing_vals_S} uses Lemma \ref{lem:norm_of_Pi}:
\begin{lemma}\label{lem:norm_of_Pi}
     The operator norm of an SEM $S \in M_{k\times n}$ is the square root of the maximal number of non-zeros in a row in $S$.
\end{lemma}
\begin{proof}
    Assume, without loss of generality, that there are $\kappa_i$, $i=1,\ldots ,k$ non-zeros in each row,  $\kappa_1 \geq \ldots \geq \kappa_k$.
    Denote the set of non-zero indeces in the $i$th row by $K_i$ ($|K_i| = \kappa_i$, $i = 1, \ldots, k$).
    Since there is only one non-zero in each column, $\sum_{i=1}^k \kappa_i = n$.
    There is a vector $v$ of unit length such that $\|Sv\| =\sqrt{\kappa_1} $.
    Let  $v = (v_1, \ldots, v_n)^T$ be such that $\sum_{i=1}^k v_i^2 = 1$. Then
    $$
    \|Sv\|_2^2 = \sum\limits_{j=1}^k\left( \sum\limits_{i \in K_j}v_{i}\right) ^2.
    $$
    Since $\max\limits_{\sum\limits_{K_j} v_i^2=\alpha_j} \sum v_i$ is achieved when $v_i = \sqrt{\frac{\alpha_j}{\kappa_j}}$ for all $ i \in K_j$, then we have
    $$
        \|Sv\|^2 \leq \sum\limits_{j=1}^k\left( \kappa_j  \sqrt{\frac{\alpha_j}{\kappa_j}} \right) ^2 = \sum\limits_{j=1}^k\left(  \sqrt{\kappa_j \alpha_j} \right) ^2 = \sum\limits_{j=1}^k \kappa_j \alpha_j.
    $$
    Since $\sum \alpha_j = 1 $ it follows that
    $
        \|Sv\|^2 \leq \kappa_1.
    $
    Thus,
    $
    \|Sv\| \leq \sqrt{\kappa_1}.
    $
\end{proof}
\begin{rem}
	In a similar way, one can show that all the singular values of $S$ are of the form $\sqrt{\kappa_i}$.
\end{rem}

\begin{proof}[Proof of Theorem \ref{thm:sing_vals_S}]
    By Lemma \ref{lem:norm_of_Pi}, the norm of $S$ is the square root of the maximal number of non-zeros ($\nnz$) in a row.
    The maximal $\nnz$  in each row is distributed as the maximum of $n$ balls thrown into $k$ urns.
    By Theorem 1 in \cite{raab1998balls}, the probability of the norm to be more than $\frac{n}{k} + \sqrt{2\frac{n}{k} \log k}$ is $o(1)$.
    Thus, the norm is bounded, with high probability, by $\sqrt{\frac{n}{k} + \sqrt{2\frac{n}{k} \log k}}$ for sufficiently large $n$
\end{proof}

\begin{theorem}[Appears 
    as Theorem 3 in \cite{nelson2013osnap} ] \label{thm:FJLT_sp}
    For any $U \in M_{m\times r}$ with orthogonal columns, if $S \in M_{l \times m}$ where $l \geq \delta ^{-1}(r^2+r)/(2\varepsilon - \varepsilon^2)^2$ is a randomly chosen SEM,  then with probability of at least $1-\delta$, the largest and smallest singular values of $SU $ are in the interval $[1-\varepsilon, 1+\varepsilon]$.

\end{theorem}

\begin{corollary} \label{cor:sing_SU}
    Let $\Omega = \Pi S$, where $\Pi \in M_{k \times l}$ is as in Theorem \ref{thm:FJLT_SRFT} and $S \in M_{l \times m}$ as in Theorem \ref{thm:FJLT_sp}. Then,
    for any $U \in M_{m\times r}$, which has orthogonal columns with high probability,
    $\|\Omega U\|_2 \leq 1.48(1+\varepsilon)$ and $\|(\Omega U)^{-1}\|_2 \leq 0.4\frac{1}{(1-\varepsilon)}$.
\end{corollary}

\begin{theorem}[Appears as Lemma 46 in \cite{clarkson2013low}] \label{thm:affine_sparse_embed}
    Let $A \in M_{m \times d}$ be of rank $r$, $B \in M_{m \times d'}$,
    and $c = d + d'$. For SEM $S \in M_{l \times m}$ and SRFT matrix $\Pi \in M_{k \times l}$, there exist $l = O(r^2 \log^6(r/\varepsilon) + r\varepsilon^{-1})$
    and $k = O(r\varepsilon^{-1} \log(r/\varepsilon))$ such that for $\Omega =\Pi S$, $\tilde{X} = \argmin_X \|\Omega(AX - B)\|_F$ satisfies $\|A \tilde{X} - B\|_F \leq (1 +
    \varepsilon) \min_X \|AX - B\|_F$ with a fixed non-zero probability. The operator $\Omega$ can be applied in $O(\nnz(A) + \nnz(B) + lc \log l)$ operations.
\end{theorem}

An improved bound appears in \cite{nelson2013osnap} and is shown to be near optimal in \cite{nelson2013sparsity}.

\section{Sparse Randomized LU}
\label{sec:sparserandlu} Similarly to the work presented in
\cite{randecomp, halko2011finding}, the key idea in the current
algorithm is that the image of $AS$ for a randomly chosen SEM $S$ is
``close'' to the image of $A$ up to an error of order $\Delta_r$. It
is shown in \cite{halko2011finding} that for each $r$ there is $k>r$
such that if $S$ is a random matrix of size $n \times k$ generated
from the set of Gaussian i.i.d. matrices, or from SRFT matrices,
then with high probability the image of $AS$ is close to the image
of $A$. More rigorously, if we denote by $Q$ an $n \times k$ matrix
with orthonormal columns that has the same image as $AS$, which is
calculated by the QR algorithm, then $\|A-QQ^*A\|_F \le
\mathcal{O}(\Delta_r)$. We show in Theorem \ref{thm:A_QQA} that this
is also true for the set of random SEM:
\begin{theorem} \label{thm:A_QQA}
    Let $A$ be an $m \times n$ matrix. Assume that $l = O(r^2 \log^6(r/\varepsilon) + r\varepsilon^{-1})$, $k = O(r\varepsilon^{-1} \log(r/\varepsilon))$,
    $\Pi \in M_{k \times l}$ is an SRFT matrix and $S \in M_{l \times n}$ is an SEM. Let $\Omega = \Pi S$ and the QR decomposition of $A\Omega^*$ is denoted by $QR$. Then, $\|A-QQ^*A\|_F<(1+\varepsilon)\Delta_r(A)$.
\end{theorem}
The proof Theorem \ref{thm:A_QQA} uses ideas similar to some in \cite{clarkson2013low}.
\begin{proof}
    First, we show that $\min\limits_{\rank X=r} \|QX - A\|_F \leq (1 + \varepsilon) \Delta_r$. Assume $A_r$ is the best rank $r$ approximation of $A$. Then, directly from this assumption, it follows that $
    \min\limits_Y\|YA_r - A\|_F = \|A_r - A\|_F = \Delta_r
    $.
    From Theorem \ref{thm:affine_sparse_embed} follows that if
    $\tilde{Y} = \argmin \|(YA_r-A)\Omega^*\|_F$, then
    $$
    \|\tilde{Y}A_r-A\|_F \leq  (1 + \varepsilon)  \min\limits_Y\|YA_r - A\|_F =  (1 + \varepsilon) \Delta_r.
    $$
    Note that
    $$
    \argmin \|(YA_r-A)\Omega^*\|_F = \argmin \|YA_r\Omega^*-A\Omega^*\|_F = A\Omega^*(A_r\Omega^*)^\dagger.
    $$
    Thus,
    \begin{equation}\label{eq:ARsomthing_A}
    \|A\Omega^*(A_r\Omega^*)^\dagger A_r-A\|_F \leq (1 + \varepsilon) \Delta_r.
    \end{equation}
    From Eq. \eqref{eq:ARsomthing_A} it follows that
    $$
    \min\limits_{\rank X=r} \|A\Omega^*X - A\|_F \leq (1 + \varepsilon)\Delta_r \mbox{, where } X \in M_{k \times n}.
    $$
    By using the fact that
    $$
    \min\limits_{X s.t. \rank X=r} \|QX - A\|_F \leq \min\limits_{X s.t. \rank X=r}\|A\Omega^*X - A\|_F
    $$
    we get
    \begin{equation}\label{eq:minQX_A}
    \min\limits_{X s.t. \rank X=r} \|QX - A\|_F \leq (1 + \varepsilon) \Delta_r.
    \end{equation}

    It follows that $\|QQ^*A - A\|_F \leq \min\limits_{X s.t. \rank X=r} \|QX - A\|_F $, which concludes the proof.
\end{proof}

Theorem \ref{thm:A_QQA} shows that $QQ^*A$ approximates  $A$ well.
Since $Q$ and $Q^*A$ are relatively small matrices and since $Q$
has orthogonal columns, then the SVD computation of $Q^*A$ is faster than
the SVD computation of $A$. Unfortunately, $Q^*$ is a dense
matrix, then the multiplication $Q^*A$ is computationally expensive. We
now show how to replace the computation of $Q^*A$ with a
multiplication of $A$ by a sparse matrix without affecting the
accuracy too much.
\begin{corollary}\label{col:A_LLA}
    Let $A$ be a $m \times n$ matrix. Assume  $l = O(r^2 \log^6(r/\varepsilon) + r\varepsilon^{-1})$, $k = O(r\varepsilon^{-1} \log(r/\varepsilon))$,
    $\Pi \in M_{k \times l}$ is an SRFT matrix and an SEM $S \in M_{l \times n}$. Denote $\Omega = \Pi S$ and the pivoted LU decomposition of $A\Omega^*$ is denoted by $PA\Omega^* = LU$. Then $\|PA-LL^\dagger PA\|_F<(1+\varepsilon)\Delta_r(A)$.
\end{corollary}
\begin{proof}
    The proof is the same as that of Theorem \ref{thm:A_QQA}. The reason that the same proof works is that $\im \, L = \im \, Q$.
\end{proof}

\begin{algorithm}[H]
    \caption{Sparse Randomized LU Decomposition}
    \textbf{Input:} $A$ matrix of size $m \times n$ to decompose; approximation rank $r < n$; $k_1 < l_1 < k_2 < l_2$ number of columns to use in the projections and the size of output matrices.\\
    \textbf{Output:} Matrices $P,Q,L,U$ such that $\Vert PAQ-LU\Vert_F \le \mathcal{O}(\Delta_r(A))$, where $P$ and $Q$
    are orthogonal permutation matrices, $L$ and $U$ are lower and upper triangular matrices, respectively.
    \begin{algorithmic}[1]
        \STATE Create a random SEM $S_1\in M_{l_1 \times n}$ and an SRFT matrix $\Pi_1\in M_{k_1 \times l_1}$. Let $\Omega_1 = \Pi_1 S_1$ be of size $k_1 \times n$.
        \STATE Compute $B = A\Omega_1^*$ ($B\in M_{m\times k_1}$).
        \STATE Compute the LU decomposition of $B$: $PB = L_1U_1$, where $L_1\in M_{m\times k_1}$ is a lower triangular matrix and   $U_1\in M_{k_1\times k_1}$ is an upper triangular matrix.
        \STATE Create a random SEM $S_2$ of size $l_2 \times m$ and an SRFT matrix $\Pi_2\in M_{k_2 \times l_2}$. Let $\Omega_2 = \Pi_2 S_2$ be of size $k_2 \times m$.
        \STATE Compute $\Omega_2 L_1$ and $(\Omega_2 L_1)^\dagger$.
        \STATE Compute the LU decomposition with right partial pivoting of $(\Omega_2 L_1)^\dagger \Omega_2 P A$ such that $(\Omega_2 L_1)^\dagger \Omega_2 P AQ =
        \tilde{L}U$.
        \STATE $L \gets L_1\tilde{L}$.
        \STATE Return $L,U,P,Q$ 
    \end{algorithmic}
    \label{alg:sparse_randomized_LU_2}
\end{algorithm}

\begin{theorem}[Correctness of the algorithm]
    Let $A$ be an $m \times n$ matrix. The sparse randomized LU decomposition of $A$ uses the integers
     $k_1 = \mathcal{O}(r\log(r)), k_2 = \mathcal{O}(r), l_1 = \mathcal{O}(r^2\log^6(r)), l_2 = \mathcal{O}(r^2)$.
      Application of Algorithm \ref{alg:sparse_randomized_LU_2} gives $PAQ \approx LU$, where $P$ and $Q$ are permutation matrices, and $L$ and $U$ are lower and upper triangular matrices,
      respectively. Then, the approximation error from the application of the sparse randomized LU decomposition is bounded by $\|LU - PAQ\|_F \leq \mathcal{O}(\Delta_r)$ with high probability.
\end{theorem}

\begin{proof}
Choose $0 < \varepsilon < 1$ ($\varepsilon$ affects the error of the
decomposition) and $0 < \delta < 1$ ($\delta$ affects the
probability that the decomposition is accurate). According to
Algorithm \ref{alg:sparse_randomized_LU_2}, $\Omega_1=\Pi_1 S_1 \in
M_{k_1\times n}$ where $\Pi_1 \in M_{k_1 \times l_1}$ is an SRFT
matrix and $S_1$ is a random SEM. The pivoted LU decomposition of
$B$ is given by $PB=L_1U_1$. Let $ k_1 =
\mathcal{O}(r\varepsilon^{-1}\log(r/\varepsilon))$ and $l_1 =
\mathcal{O}(r^2\log^6(r/\varepsilon)+ r\varepsilon^{-1})$.  Then
from Corollary \ref{col:A_LLA} it follows that $\|PA-L_1L_1^\dagger
PA\|_F<(1+\varepsilon)\Delta_r$. Let
$$
l_2 \geq \delta^{-1}(r^2+r)/(2\varepsilon-\varepsilon^2)^2, k_2 \geq 4\left[\sqrt{r} + \sqrt{8\log(rl_2)}\right]^2 \log r.
$$
Then, by Corollary \ref{cor:sing_SU}, with high probability,
$\Omega_2 L_1$ is left invertible. Thus,
$$\|L_1L_1^\dagger PA - PA\|_F = \|L_1(\Omega_2 L_1)^\dagger (\Omega_2 L_1)L_1^\dagger P A - P A\|_F. $$
Next, we bound $\|L_1(\Omega_2 L_1)^\dagger (\Omega_2 L_1)L_1^\dagger P A - P A\|_F$ by the following:

\begin{equation}\label{eq:bound1proof}
    \begin{split}
        \|L_1& (\Omega_2 L_1)^{-1}\Omega_2 PA- PA\|_F= \|L_1(\Omega_2 L_1)^{-1}\Omega_2 PA - L_1(\Omega_2 L_1)^{-1}(\Omega_2 L_1)L_1^\dagger PA \\
        &+L_1(\Omega_2 L_1)^{-1}(\Omega_2 L_1)L_1^\dagger PA - PA\|_F\\
        &=\|L_1(\Omega_2 L_1)^{-1}\Omega_2 (PA -L_1L_1^\dagger PA) + L_1L_1^\dagger PA - PA\|_F \\
        & \leq\|L_1(\Omega_2 L_1)^{-1}\Omega_2 (PA-L_1L_1^\dagger PA)\|_F  +  \|L_1L_1^\dagger PA - PA\|_F \\
        & \leq \|L_1(\Omega_2 L_1)^{-1}\Omega_2 \|_2\|PA-L_1L_1^\dagger PA\|_F +\|L_1L_1^\dagger PA - PA\|_F \\
        &=(\|L_1(\Omega_2 L_1)^{-1}\Omega_2 \|_2 + 1)\|PA-L_1L_1^\dagger PA\|_F.
     \end{split}
\end{equation}
Let $L_1 = U\Sigma V^*$ be the SVD of $L_1$, where $U\in M_{m\times
k_1}$,  $\Sigma \in M_{k_1\times k_1}$, and $V \in M_{k_1\times
k_1}$. Then
\begin{equation}\label{eq:bound2proof}
\begin{array}{lll}
    \|L_1(\Omega_2 L_1)^{-1}\Omega_2 \|_2 & = & \|U \Sigma V^*(\Omega_2  U \Sigma V^*)^{-1}\Omega_2 \|_2 \\
    & = & \|U \Sigma V^*(\Sigma V^*)^{-1}(\Omega_2 U )^{-1}\Omega_2 \|_2 \\
    & = & \|U (\Omega_2 U )^{-1}\Omega_2\|_2 \\
    & = & \| (\Omega_2 U )^{-1}\Omega_2\|_2 \\
    & \leq & \| (\Omega_2 U )^{-1}\|_2 \|\Omega_2\|_{2} .
\end{array}
\end{equation}
By combining Eqs. \eqref{eq:bound1proof} and \eqref{eq:bound2proof}
with Corollary \ref{cor:sing_SU} and Theorem \ref{thm:sing_vals_S},
we get
$$
\begin{array}{lll}
\|L_1(\Omega_2 L_1)^{-1}\Omega_2 PA- PA\|_F & \leq  &(\frac{C(n,k_2)}{0.4(1-\varepsilon)} + 1)\|PA-L_1L_1^\dagger PA\|_F \\ &\leq &1.48(1+\varepsilon)\left(\frac{C(n,k_2)}{0.4(1-\varepsilon)} + 1\right) \Delta_r.

\end{array}
$$
From Algorithm \ref{alg:sparse_randomized_LU_2}, obtain
$$\|LU - PAQ\|_F =\|L_1 \tilde{L}U Q^* - PA\|_F =\|L_1(\Omega_2 L_1)^{-1}\Omega_2 PA - PA\|_F \leq \mathcal{O}(\Delta_r),$$
which completes the proof.
\end{proof}

\subsection{Algorithm Complexity}

Denote $m,n,r,k_1, k_2, l_1, l_2$ as in Algorithm \ref{alg:sparse_randomized_LU_2}. Assume, without loss of generality, that $m \geq n$. Then
\begin{enumerate}
    \item $\Omega_1$ construction takes  $\mathcal{O}(n + l_1 k_1)$ operations.
    \item $B=A\Omega_1^*$ computation takes $\mathcal{O}(mn+ml_1\log(k_1))$ operations.
    \item Computation of the pivoted LU decomposition of $B$ takes $\mathcal{O}(mk_1^2)$ operations.
    \item $\Omega_2$ construction takes  $\mathcal{O}(m + l_2 k_2)$ operations.
    \item $\Omega_2 L_1$ and $(\Omega_2 L_1)^{\dagger}$ computation takes $\mathcal{O}(m k_1 + k_1 l_2 \log(k_2))$ and $\mathcal{O}(k_2k_1^2)$ operations respectively.
    \item $(\Omega_2 L_1)^{\dagger}\Omega_2 PA $ computation takes $\mathcal{O}(mn+nl_2\log(k_2)+k_2k_1n)$ operations.
    \item LU decomposition of  $(\Omega_2 L_1)^{\dagger}\Omega_2 PA$ takes $\mathcal{O}(k_2^2n)$ operations.
    \item $L = L_1\tilde{L}$ computation takes $\mathcal{O}(mk_1^2)$ operations.
\end{enumerate}
This sums up to a total complexity of
$$
\mathcal{O}\left(mn  +  mk_1^2  +  nk_2^2  +  m l_1\log(k_1)  +  nl_2\log(k_2)  +  k_1l_2\log(k_2)  \right),
$$
and the complexity of the decomposition of a sparse matrix $A$ is 
$$
\mathcal{O}\left(\nnz(A)  +  mk_1^2  +  nk_2^2  +  m l_1\log(k_1)  +  nl_2\log(k_2)  +  k_1l_2\log(k_2)  \right).
$$

\section{Numerical Results}
\label{sec:numerical_results} In this section, the performance of
the algorithm is evaluated. The algorithm is implemented in MATLAB
using complex matrices. The Sub-sampled Randomized Hadamard Transform (SRHT) \cite{tropp2011improved} is used with real matrices instead of using the SRFT
matrix to achieve an efficient computation.

\subsection{Numerical rank growth}
In this experiment, we consider a matrix of size $n = 5000$ where
its  numerical rank changes between 50 to 900, i.e., the first $r$
singular values are 1 and the other are exponentially decaying from
$e^{-10}$ to $e^{-200}$. As shown in Figure \ref{fig:const_n__step_sigma}, Algorithm \ref{alg:sparse_randomized_LU_2} results in an approximation of
the same order as the numerical rank, up to a small error.

    \begin{figure}[H]
        \makebox[\textwidth][c]{
	        \begin{subfigure}[b]{0.6\textwidth}
	            \includegraphics[width=1\textwidth]{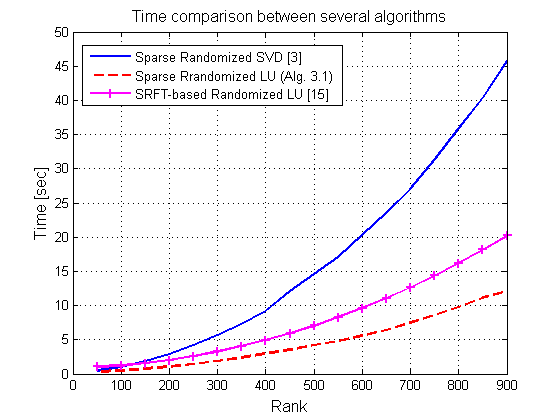}
	            \caption{}
	            \label{fig:const_n__step_sigma_time}
	        \end{subfigure}%
	        ~
	        \begin{subfigure}[b]{0.6\textwidth}
	            \includegraphics[width=1\textwidth]{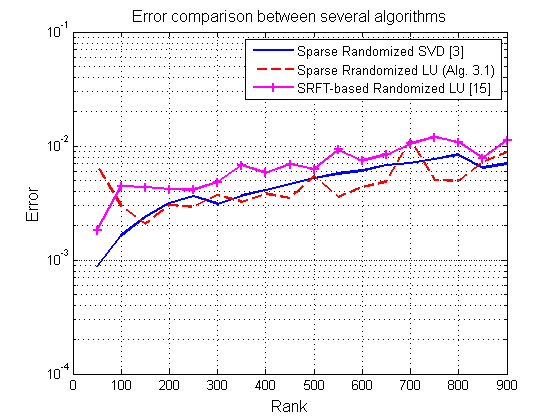}
	            \caption{}
	            \label{fig:const_n__step_sigma_err}
	        \end{subfigure}
	    }

        \caption{Results from the approximation of a matrix of size $5000 \times 5000$ with different numerical ranks.
         The numerical rank is shown on the x-axis. (a) the y-axis denotes the time  each algorithm takes. (b) the y-axis denotes the error of each algorithm. }
        \label{fig:const_n__step_sigma}
    \end{figure}

\subsection{Improving the accuracy for a fixed matrix}
In this experiment, we consider a matrix of size $n = 5000$ with singular values that decay exponentially from 1 to $e^{-100}$.
 We compute the $r$-th rank approximation by increasing $r$.

    \begin{figure}[H]
        \makebox[\textwidth][c]{
		    \begin{subfigure}[b]{0.6\textwidth}
			    \includegraphics[width=1\textwidth]{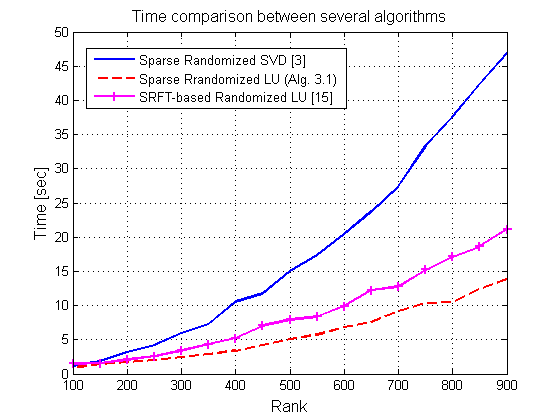}
	            \caption{}
	            \label{fig:const_n__const_sigma_time}
	        \end{subfigure}%
	        ~
	        \begin{subfigure}[b]{0.6\textwidth}
			    \includegraphics[width=1\textwidth]{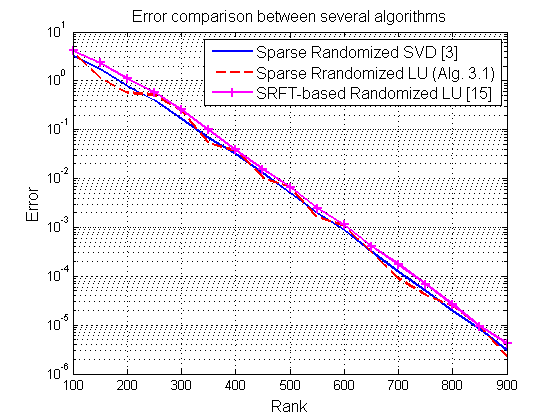}
	            \caption{}
	            \label{fig:const_n__const_sigma_err}
	        \end{subfigure}
	    }
        \caption{Results from the  approximation a matrix of size $5000 \times 5000$ with exponentially decaying singular values.
        The approximation rank is shown on the x-axis. (a) the y-axis denotes the time each algorithm takes. (b) the y-axis denotes the error of each algorithm.}
        \label{fig:const_n__const_sigma}
    
    \end{figure}

\subsection{Running on GPU}
The Sparse randomized LU decomposition (Algorithm \ref{alg:sparse_randomized_LU_2}) can be
fully parallelized to run efficiently on a GPU card and on a
distributed computing system such as Hadoop or Spark. In the following test, a $5000 \times 5000$
random matrix was processed in double precision on a GPU card using
the MATLAB's GPU interface. MATLAB 2015a enables us to apply 
certain sparse matrices operations to the GPU. GTX Titan Black GPU card was used. Figure \ref{fig:gpu_vs_cpu} compares the running
time between GPU and CPU.

\begin{figure}[H]
    \centering
    \includegraphics[width=0.75\textwidth]{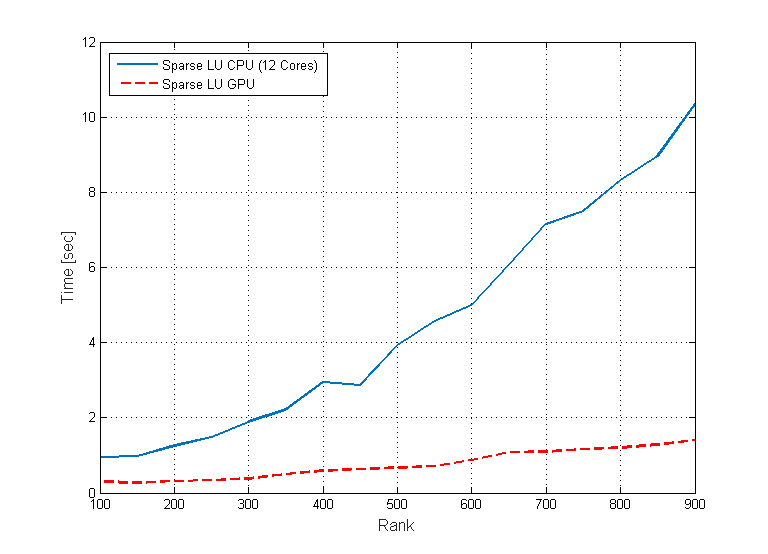}
    \caption{Running time on GPU vs. CPU of the randomized sparse LU decomposition (Alg. \ref{alg:sparse_randomized_LU_2})}
    \label{fig:gpu_vs_cpu}
\end{figure}

\section*{Conclusion}
\label{sec:conclusion} In this paper, the Sparse--Randomized--LU
algorithm is presented. This algorithm utilizes sparse random
projections that are combined with FFT--based projections for
computing low rank LU matrix decompositions. The proposed technique
was analyzed theoretically to achieve asymptotic bounds. The
conducted numerical experiments  compare the performance of the
algorithm to other  algorithms such as sparse SVD and fast
randomized LU.
\section*{Acknowledgment}
This research was partially supported by the Israeli Ministry of
Science \& Technology (Grants No. 3-9096, 3-10898), US-Israel
Binational Science Foundation (BSF 2012282), Blavatnik Computer
Science Research Fund and Blavatink ICRC Funds.

\bibliographystyle{siam}
\bibliography{SparseRandomizedLU}
\end{document}